\documentclass[letterpaper, 10 pt, journal, twoside]{ieeetran}  

\IEEEoverridecommandlockouts 

\usepackage{graphicx} 
\usepackage{amsmath} 
\usepackage{amssymb}  

\newcommand{\eome}{{\rm \nabla\kern-.6em \nabla}}
\newcommand{\diag}{\mathrm{diag}}
\newtheorem{definition}{Definition}
\newtheorem{theo}{Theorem}
\newtheorem{lemma}{Lemma}
\newtheorem{remarkTh}{Remark}

\usepackage[bb=boondox]{mathalfa}
\newcommand{\Chi}{\mathfrak{X}}

\newcommand{\supeq}{\geqslant}
\newcommand{\infeq}{\leqslant}
\newcommand{\V}{\mathcal{V}}

\newcommand{\R}{\mathbb{R}}
\renewcommand{\S}{\mathbb{S}}

\newenvironment{proof}[1][Proof]{\textbf{#1.} }{\ \rule{0.5em}{0.5em}}
\newenvironment{remark}[0]{\begin{remarkTh}}{\hfill$\square$\end{remarkTh}}

\DeclareMathOperator\He{He}
\DeclareMathOperator\sign{sign}

\renewcommand\epsilon{\varepsilon}

\title{\LARGE \bf Exponential Lyapunov Stability Analysis of a Drilling Mechanism}

\author{Matthieu~Barreau, Alexandre~Seuret and Fr\'ed\'eric~Gouaisbaut
\thanks{M. Barreau, A. Seuret, F. Gouaisbaut are with LAAS - CNRS, Universit\'e de Toulouse, CNRS, UPS, France e-mail: (mbarreau,aseuret,fgouaisb@laas.fr).}
}

\begin{document}

\maketitle
\thispagestyle{empty}
\pagestyle{empty}

\begin{abstract}
This article deals with the stability analysis of a drilling system which is modelled as a coupled ordinary differential equation / string equation. 
The string is damped at the two boundaries but leading to a stable open-loop system. The aim is to derive a linear matrix inequality ensuring the exponential stability with a guaranteed decay-rate of this interconnected system. A strictly proper dynamic controller based on boundary measurements is proposed to accelerate the system dynamics and its effects are investigated through the stability theorem and simulations. It results in an efficient finite dimension controller which subsequently improves the system performances. 
\end{abstract}

\section{Introduction}

Many physical situations like string-payloads \cite{he2014adaptive} or drilling systems \cite{bresch2014output} are modeled by infinite dimensional systems. They are, in their fundamentals, related to a Partial Differential Equation (PDE) and consequently, their stability analysis and control are not straightforward and has been under active research during the last decade.

A drilling mechanism is within this class of systems. It is used in the industry to pump oil deep in the soil. This physical system is subject to torsion and radial deformation due to the torque applied on one boundary of the pipe. This system is usually modeled by a coupled Ordinary Differential Equation (ODE) / string equation. These heterogeneous equations appear naturally when the torsional motion of the pit is coupled with the axial deformation of the pipe \cite{challamel2000rock}. Moreover, as there is friction all along the pipe, it leads to a complex system made up of two non-linear equations. The commonly used methodology to control this system is the backstepping. 

The aim is to use a control to transform the problem into a target system with the desired properties. Then, using a Lyapunov approach for example, the stability can be proven. This has been widely used in \cite{bresch2014output,krstic2009delay,krstic2008boundary,Wu20142787}. There are many advantages because it provides a Lyapunov functional useful for a robustness analysis for example but it also provides a very accurate control as it mostly depends on the target system. But the calculations are tedious and lead to an infinite dimension control law which may be subjected to implementation issues.

Coming from the stability analysis of time-delay systems, a new method based on Linear Matrix Inequalities (LMIs) seems to be promising. As time-delay systems are a particular case of infinite dimension systems \cite{fridman2014}, it is possible to extend the methodology described in \cite{seuret:hal-01065142} to other systems. It relies on a Lyapunov functional and a state extension using projections of the infinite dimensional state on a basis of orthonormal polynomials. The key result is based on an extensive use of Bessel inequality. It has been successfully applied to transport equations in \cite{SAFI20171}, to the heat equation \cite{baudouinHeat} and to the wave equation also \cite{besselString}. 

In this paper, we focus on the exponential stability analysis of a linearized drilling mechanism as described in \cite{marquez2015analysis} with the previous methodology. First, we explain the problem and discuss the existence of a solution. Then, an exponential stability result is provided. The theorem ensures the exponential stability with a guaranteed decay-rate. Some necessary conditions are drawn from the LMI condition and then, an example using physical values is provided. A control law is also derived to show the effectiveness of the method.

\textbf{Notations:} In this paper, $\mathbb{R}^+ = [0, +\infty)$ and $(x,t) \mapsto u(x,t)$ is a multi-variable function from $[0,1] \times \mathbb{R}^+$ to $\mathbb{R}$.
The notation $u_t$ stands for $\frac{\partial u}{\partial t}$. We also use the notations $L^2 = L^2((0, 1); \mathbb{R})$ and for the Sobolov spaces: $H^n = \{ z \in L^2; \forall m \infeq n, \frac{\partial^m z}{\partial x^m} \in L^2 \}$. 
The norm in $L^2$ is $\|z\|^2 =  \int_{\Omega} |z(x)|^2  dx = \left<z,z\right>$. 
For any square matrices $A$ and $B$, the operations '$\text{He}$' and '$\text{diag}$' are defined as follow: $\text{He}(A) = A + A^{\top}$ and $\text{diag}(A,B) = \left[ \begin{smallmatrix}A & 0\\ 0 & B \end{smallmatrix} \right]$.
A positive definite matrix $P \in \R^{n \times n}$ belongs to the set $\S^n_+$ and $P \succ 0$.

\section{Problem Statement}

\subsection{Modeling of the drilling process}

A drilling mechanism was first modeled in \cite{fridman2010bounds} using the work of \cite{challamel2000rock}. This system described in Figure~\ref{fig:drilling} is the result of a coupling between a radial deformation and an axial movement. This coupling was later modeled in \cite{marquez2015analysis,saldivar2016control} by the following nonlinear model for $x \in (0, 1)$ and $t > 0$:
\begin{equation} \label{eq:problem}
	\hspace{-0.05cm}
	\left\{
		\begin{array}{ll}
			z_{tt}(x,t) = c^2 z_{xx}(x,t) - d z_t(x,t),\\
			z_x(0,t) = g \left( z_t(0,t) - \tilde{u}_1(t) \right), \\
			z_x(1,t) = -h z_{tt}(1,t) - k z_t (1,t)  - q T_{nl}( z_t(1,t) ), \\
			\dot{Y}(t) = A Y(t) + B \tilde{u}_2(t) + E_1 z_t(1,t) + E_2 T_{nl} (z_t(1,t)), \\
		\end{array}
	\right.
	\hspace{-0.9cm}
\end{equation}
with initial condition $z(\cdot,0) = z^0$, $z_t(\cdot,0) = z_t^0$ on $(0, 1)$ and $Y(0) = Y^0$. In this model, $z$ is the twist angle and it propagates along the pipe following a damped wave equation of speed $c$ and internal damping $d$. Since the internal damping stabilizes the system, in this study, we consider the worst case scenario with $d = 0$ like in \cite{fridman2010bounds}. A similar work can be done with $d > 0$ but leads to more tedious calculation and is then omitted.
There are two boundary conditions at $x= 0$ and $x = 1$. At $x = 0$, a rotary table whose speed is controlled by the input $\tilde{u}_1$ allows to twist the pipe. Furthermore, the boundary damping with a coefficient $g$ at $x = 0$ represents a viscous friction torque.\\
\begin{figure}
	\centering
	\includegraphics[width=8cm]{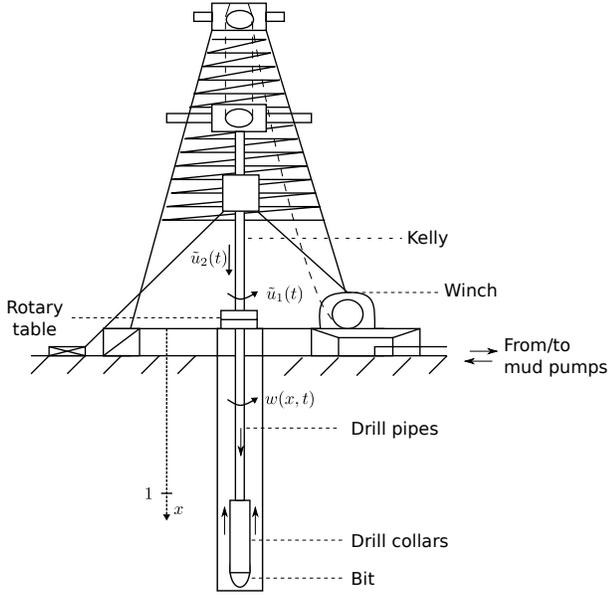}
	\caption{Schematic of a drilling mechanism originally taken from \cite{saldivar2016control}. Data corresponding to physical vaues are given in Table~\ref{tab:values}.}
	\label{fig:drilling}
\end{figure}
\!\!\!The drilling pit is located at $x = 1$. When drilling, an external torque applies at this boundary and the momentum equation leads to a second order in time boundary condition. The term $T_{nl}$ is a non-linear function related to the change of torque and given below. To simplify the system as done in \cite{fridman2010bounds}, we consider the equation at the bottom of the pipe to be only a first order boundary damping, then $h = 0$. \\
The axial deformation is modeled by a finite dimensional equation as noted in \cite{challamel2000rock}. This equation is related to the axial deformation of the pipe. In \cite{saldivar2016control}, a second order damped harmonic oscillator is used because it models a mass subject to a force for small vibrations. The control at $x = 0$ for the axial position is $t \mapsto \tilde{u}_2(t)$ and corresponds to the force needed in the system to drill. Denoting by $y$ the axial bit position and by $\Gamma_0$ the rate of penetration, $Y(t) = \left[ y(t) - \Gamma_0 t \ \ \dot{y}(t) - \Gamma_0 \right]^{\top} \in \mathbb{R}^2$ represents the axial position error and axial velocity error, leading to the last equation in \eqref{eq:problem}.
\begin{remark} Note that this model does not take into account a coupling between torsion and axial deformation but more a cascaded effect between them.\end{remark}
The parameters $c,g,k,q,A_{21}, A_{22}, b, e_1$ and $e_2$ are physical parameters given in \cite{saldivar2016control} and reported in Table~\ref{tab:values}. The matrices have the following structure:
\[
	\begin{array}{cccc}
		A = \left[ \begin{smallmatrix} 0 & 1 \\ A_{21} & A_{22} \end{smallmatrix} \right], & B = \left[ \begin{smallmatrix} 0 \\ b \end{smallmatrix} \right], & 
		E_1 = \left[ \begin{smallmatrix} 0 \\ e_1 \end{smallmatrix} \right], & E_2 = \left[ \begin{smallmatrix} 0 \\ e_2 \end{smallmatrix} \right].
	\end{array}	
\]

The aim is to design control laws $\tilde{u}_1$ and $\tilde{u}_2$ such that the angular speed $z_t(1,t)$ in system \eqref{eq:problem} converges to the desired angular velocity $\Omega_e$ and $Y$ to $0$. Without loss of generality, we assume $\Omega_e > 0$.

In \cite{challamel2000rock,saldivar2016control}, the nonlinear part of the torque is described by the following equations for $\theta \in \R$:
\begin{equation} \label{eq:Tnl}
	\left\{
		\begin{array}{l}
			T_{nl}(\theta) = W _{ob} R_{b} \mu_b(\theta) \sign(\theta), \\
			\mu_b(\theta) = \mu_{cb} + \left(\mu_{sb} - \mu_{cb}\right) e^{-\gamma_b |\theta|}.
		\end{array}
	\right.
\end{equation}

Considering $\Omega_e \gg 0$, then $e^{- \gamma_b \Omega_e}$ is small and $T_{nl}$ is linearized around $\Omega_e$ as follows:
\begin{equation} \label{eq:T}
	T_{nl}(z_t(1,t)) \simeq W _{ob} R_{b} \mu_{cb} = T^e.
\end{equation}
\begin{remark} This approximation prevents from the stick-slip effect which is the main problem that occurs when dealing with drilling pipes for small $\Omega_e$. This work can be seen as a preliminary version of an extended one considering the non-linearity. \end{remark}
That leads to an approximated linear system defined for $t \supeq 0$ with the same initial conditions and $x \in (0, 1)$:
\begin{equation} \label{eq:problem}
	\hspace{-0.05cm}
	\left\{
		\begin{array}{l}
			w_{tt}(x,t) = c^2 w_{xx}(x,t), \\
			w_x(0,t) = g \left( w_t(0,t) - \tilde{u}_1(t) \right),  \\
			w_x(1,t) = - k w_t (1,t) - q T^e \\
			\dot{Y}(t) = A Y(t) + B \tilde{u}_2(t) + w_t(1,t) E_1 - T^e E_2.
		\end{array}
	\right.
	\hspace{-1cm}
\end{equation}

It is possible to use the Riemann coordinates to simplify the writing of this system using the following variable: $\tilde{\chi}(x,t) = \left[ \begin{smallmatrix} w_t(x,t) + c w_x(x,t) \\ w_t(1-x,t) - c w_x(1-x,t) \end{smallmatrix} \right]$.
The system becomes for $t \supeq 0$:
\begin{equation} \label{eq:linProblem}
	\left\{
		\begin{array}{l}
			\tilde{\chi}_{t}(x,t) = c \tilde{\chi}_{x}(x,t), \quad \quad x \in (0, 1), \\
			\left[ \begin{smallmatrix} 1-cg & 0 \\ 0 & 1-ck \end{smallmatrix} \right] \tilde{\chi}(0,t) = \left[ \begin{smallmatrix} 0 & 1+cg \\ 1+ck & 0 \end{smallmatrix} \right] \tilde{\chi}(1,t) + \left[ \begin{smallmatrix} - 2 cg \tilde{u}_1(t) \\ 2 c q T^e \end{smallmatrix} \right], \\
			\dot{Y}(t) = A Y(t) + B \tilde{u}_2(t) + \tilde{E}_1 \left[ \begin{smallmatrix} \tilde{\chi}(0,t) \\ \tilde{\chi}(1,t) \end{smallmatrix} \right] - T^e E_2,
		\end{array}
	\right.
\end{equation}
with $\tilde{E}_1 = \frac{1}{2} E_1 \left[ \begin{smallmatrix} 0 & 1 & 1 & 0 \end{smallmatrix}\right]$. The stability of system~\eqref{eq:linProblem} implies the stability of \eqref{eq:problem} and then the study focuses on system \eqref{eq:linProblem}.


Assuming $(\tilde{\chi}^e, Y^e)$ is an equilibrium point of system \eqref{eq:linProblem}, it satisfies $\tilde{\chi}^e_{t} = 0$,  $w^e_t = \Omega_e$ and $\dot{Y}^e = 0$. Therefore, a feedforward open-loop control is introduced as:
\begin{equation} \label{eq:feedforward}
	\tilde{u}_1^e = \Omega_e \left( 1 + \frac{k}{g} \right) + \frac{q}{g} T^e,  \quad \tilde{u}_2^e = \frac{T^e e_2 - \Omega_e e_1}{b}.
\end{equation}

Introducing the error variables $\chi(x,t) = \tilde{\chi}(x,t) - \tilde{\chi}^e(x)$, $u_1(t) = \tilde{u}_1(t) - \tilde{u}_1^e$ and $u_2 (t)= \tilde{u}_2(t) - \tilde{u}_2^e$,  the aim is to show the exponential stability of $\chi$ to $0$ in order to get $w_t \to \Omega_e$ and $\| Y \| \to 0$. The inputs $u_1$ and $u_2$ are assumed to be the results of a strictly proper dynamic controller whose inputs are $w_t(0,t), w_t(1,t)$ and $Y$. That means that the measurements are these three variables but it is not possible to apply exactly $w_t(1)$ or $w_t(0)$, corresponding to the situation where the actuator is bandwidth limited for instance. This assumption is important as the wave can be seen as a neutral system \cite{barreauInputOutput} and using directly $w_t$ means that we can affect directly the neutral part. This phenomena is known to be absolutely non-robust \cite{hale2001effects} to small delay for example. Assuming the controller is of order $n$, it is written for $t \supeq 0$:
\[
	\left\{
		\begin{array}{cl}
			\dot{X}_c(t) \!\!\!\!& = A_c X_c(t) + B_{c1} Y(t) + B_{c2} \left[ \begin{smallmatrix} w_t(0,t) \\ w_t(1,t) \end{smallmatrix} \right], \\
			u_1(t) \!\!\!\!& = C_1 \left[ \begin{smallmatrix} X_c(t) \\ Y(t) \end{smallmatrix} \right], \\
			u_2(t) \!\!\!\!& = C_2 X_c(t) + K Y(t).
		\end{array}
	\right.
\]
with $C_1, C_2 \in \mathbb{R}^{1 \times (n+2)}, A_c \in \mathbb{R}^{n \times n}, B_{c1}, B_{c2} \in \mathbb{R}^{n,2}$ and $K \in \mathbb{R}^{1 \times 2}$. The closed-loop system in Riemann coordinates can be rewritten as:
\begin{equation} \label{eq:linWaveProblem}
	\hspace{-0.15cm}
	\left\{
		\begin{array}{l}
			\chi_{t}(x,t) = c \chi_{x}(x,t), \\
			\left[ \begin{smallmatrix} 1-cg & 0 \\ 0 & 1-c k \end{smallmatrix} \right] {\chi}(0,t) = \left[ \begin{smallmatrix} 0 & 1+cg \\ 1+c k & 0 \end{smallmatrix} \right] {\chi}(1,t) - \left[ \begin{smallmatrix} 2 c g C_1 X(t) \\ 0 \end{smallmatrix} \right], \\
			\dot{X}(t) = \tilde{A} X(t) + \tilde{B} \left[ \begin{smallmatrix} \chi(0,t) \\ \chi(1,t) \end{smallmatrix} \right],
		\end{array}
	\right.
\end{equation}
with initial conditions $\chi(x, 0) = \chi_{0}(x), X(0) = X^0$, $X^{\top} \!=\! \left[ \begin{smallmatrix} X_c^{\top} & Y^{\top} \end{smallmatrix} \right]^{\top}$ and 
\[
	\tilde{A} \!=\! \left[ \begin{matrix} A_c & B_{c1} \\ B C_2 & A \!+\! BK \end{matrix} \right]\!\!,  \quad \quad \tilde{B} \!=\! \frac{1}{2}\left[ \begin{matrix} B_{c2} \\  \begin{array}{cc} E_1 & 0_{2, 1} \end{array} \end{matrix} \right] \left[ \begin{smallmatrix} 1 & 0 \\ 0 & 1 \\ 0 & 1 \\ 1 & 0 \end{smallmatrix} \right]^{\top}\!\!.
\]
\begin{remark} A similar control law is proposed in \cite{ATJECC} but the stability is dealt using another Lyapunov functional. \end{remark}
\begin{remark} From now on, to ease the reading, the parameter $t$ may be omitted and ${\chi}$ refers to a solution of \eqref{eq:linWaveProblem}. \end{remark}

\subsection{Existence and uniqueness}

The existence and uniqueness follows the same lines than in \cite{besselString}. Define the following set: $\mathcal{H}^m = \mathbb{R}^{n+2} \times H^m \times H^m$ with $m \in \mathbb{N}$. The space $\mathcal{H} = \mathcal{H}^0$ can be equipped with the following norm:
\[
	\begin{array}{ll}
		\forall (X, \chi) \in \mathcal{H}, \quad \|(X, \chi)\|_{\mathcal{H}} \!\!\!\!& = |X|^2 + \frac{1}{2} \|\chi\|^2 \\
		& = |X|^2 + c^2 \| w_x \|^2 + \| w_t \|^2.
	\end{array}
\]

Using the operator notation \cite{tucsnak2009observation}, system \eqref{eq:linWaveProblem} is formulated as follows:
\[
	T \left( \begin{smallmatrix} X \\ \chi \end{smallmatrix} \right) = \left( \begin{matrix} \tilde{A} X + \tilde{B} \left[ \begin{smallmatrix} \chi(0) \\ \chi(1) \end{smallmatrix} \right] \\ c \chi_x \end{matrix} \right),
\text{ and } T: \mathcal{D}(T) \to \mathcal{H},
\] with
\begin{multline*}
	\mathcal{D}(T) = \left\{ (X, \chi) \in \mathcal{H}^1, \left[ \begin{smallmatrix} 1-cg & 0 \\ 0 & 1-c k \end{smallmatrix} \right] \chi(0) = \right. \\
	\left. \left[ \begin{smallmatrix} 0 & 1+cg \\ 1+c k & 0 \end{smallmatrix} \right] \chi(1) - \left[ \begin{smallmatrix} 2 c g C_1 X \\ 0 \end{smallmatrix} \right] \right\}.
\end{multline*}

The existence of a continuous solution for $(X^0, \chi_0) \in \mathcal{D}(T)$ is ensured by applying Lumer-Philips theorem (for example in \cite[p.103]{tucsnak2009observation}) whose conditions are recalled below:
\begin{enumerate}
	\item there exists a function $V:\mathcal{H} \to \mathbb{R}^+$ such that its derivative along the trajectories of  \eqref{eq:linWaveProblem} is negative;
	\item there exists $\lambda$ sufficiently small such that $\mathcal{D}(T) \subseteq \mathcal{R}(\lambda I - T)$ where $\mathcal{R}$ is the range operator. 
\end{enumerate}

The first condition relies on the existence of a Lyapunov functional and is therefore the subject of the following part. The second statement needs some calculations very similar to the one conducted in \cite{besselString} or \cite{morgul1994dynamic}. For a given $\lambda > 0$, let $(r,f) \in \mathcal{D}(T)$, the aim is to prove the existence of $(X, \chi) \in \mathcal{D}(T)$ satisfying the following for $x \in (0, 1)$:
\[
	\left\{
		\begin{array}{l}
			\lambda X - \tilde{A} X - \tilde{B} \left[ \begin{smallmatrix} \chi(0) \\ \chi(1) \end{smallmatrix} \right] = r, \\
			\lambda \chi(x) - c \chi_x(x) = f(x). \\
		\end{array}
	\right.
\]

That leads to $\chi(x) = k_1 e^{\lambda\frac{x}{c}} + F(x)$ with $F(x) = c^{-1} \int_0^x e^{\lambda \frac{x-s}{c}} f(s) ds \in H^1$ and $k_1 = \diag(k_{11}, k_{12})$, $k_{11}, k_{12} \in \mathbb{R}$. Using the boundary conditions, we get a system of two equations:
\[
		\begin{array}{l}
			(1-cg) k_1 = k_2 e^{\frac{\lambda}{c}}(1+cg)(A+F(1)) - \frac{2cg}{\lambda} C_1 X, \\
			(1-c k) k_2 = k_1 e^{\frac{\lambda}{c}}(1+c k)(A+F(1))
		\end{array}
\]
Since there exists a $\lambda$ such that $\tilde{A} + \tilde{B} \left[ \begin{smallmatrix} \chi(0) \\ \chi(1) \end{smallmatrix} \right]$ is not the null matrix, then this system has a unique solution for a given $X$ that ends the proof of existence.

\section{Exponential Stability of the Drilling Pipe}

\subsection{Main result}

The main result of this paper is the $\alpha$-stability criterion for system \eqref{eq:linWaveProblem} expressed in terms of LMIs, therefore easily tractable. Let us first define the $\alpha$-stability.
\begin{definition} System \eqref{eq:linWaveProblem} is $\alpha$-stable (or exponentially stable with a decay-rate of at least $\alpha$) with respect to the norm $\| \cdot \|_{\mathcal H}$ if there exists $\gamma \supeq 1$ such that the following holds for $(X^0, \chi_0)$ the initial condition:
\[
	\|(X(t), \chi(\cdot,t))\|_{\mathcal H} \infeq \gamma \|(X^0, \chi_0 )\|_{\mathcal H} e^{- \alpha t}.
\]
\end{definition}
Considering this definition, we propose a stability theorem for system~\eqref{eq:linWaveProblem}.
\begin{theo} \label{sec:theoLin}
Let $N > 0$. Assume there exists $P_N \in \mathbb{S}^{n+2+2(N+1)}_+$, $R, S \in \mathbb{S}^2_+$ such that the following LMI holds:
\begin{equation} \label{eq:LMI}
	\Psi_{N,\alpha} - c R_N \prec 0,
\end{equation}
with
\begin{equation} \label{eq:defTheo}
	\begin{array}{cl}
		\!\!\!\!\!\! \Psi_{N,\alpha} \!\!\!\!& = \He((Z_N + \alpha F_N)^{\top} P_N F_N) - c G_N^{\top} S G_N \\
		& \hfill \!\!\!\!\!\! + c H_N^{\top} \left( S + R \right) H_N e^{\frac{2 \alpha}{c}},\\
		\!\!\!\!\!\! F_N \!\!\!\! &= \left[ \begin{matrix} I_{n+2+2(N+1)} & 0_{n+2+2(N+1), 2} \end{matrix} \right], \\
		\!\!\!\!\!\! Z_N \!\!\!\! &= \left[ \begin{matrix} \mathcal{N}_N^{\top} & \mathcal Z_{N}^{\top} \end{matrix} \right]^{\!\top}\!\!\!,  
		\quad \mathcal{N}_N = \left[ \begin{matrix} \tilde{A} & 0_{n+2,2(N+1)} & \tilde{B} \end{matrix} \right], \\
		\!\!\!\!\!\! \mathcal Z_{N} \!\!\!\! &= c \mathbb{1}_N H_{N}\! -\! c\bar{\mathbb{1}}_N G_{N}\! -\! \left[\begin{matrix}  0_{2(N + 1),n+2} &\!\!\!\! L_{N} &\!\!\!\! 0_{2(N + 1), 2}\end{matrix} \right], \\
	\end{array}
\end{equation}
\begin{equation*}
	\begin{array}{cl}
		\!\!\!\!\!\! G_{N} \!\!\!\! &= \left[ \begin{matrix} \begin{smallmatrix} - c g C_1 \\  0_{1,n+2} \end{smallmatrix} & 0_{2, 2(N+1)} & G \end{matrix} \right], \quad G = \left[ \begin{smallmatrix} 0 & 1 + cg \\  1 + c k & 0 \end{smallmatrix} \right], \\
		\!\!\!\!\!\! H_{N} \!\!\!\! &= \left[ \begin{matrix} \begin{smallmatrix} 0_{1,n+2} \\  c g C_1 \end{smallmatrix} & 0_{2, 2(N+1)} & H \end{matrix} \right],  \quad H = \left[ \begin{smallmatrix} 1 - c k & 0 \\  0 & 1 - c g \end{smallmatrix} \right], \\
		\!\!\!\!\!\! R_N  \!\!\!\!  &= \diag(0_n, R, 3R, \cdots, (2N+1)R, 0_{2}), \\
	\end{array}
\end{equation*}
\begin{equation*}
		\begin{array}{rcllcl}
		\!\!\! L_N\!=\!\left[\!\begin{smallmatrix} 
		\ell_{0,0}I_2 & \cdots& 0_2 \\ 
		\vdots & \ddots &\vdots\\ 
		\ell_{N,0}I_2 & \cdots & \ell_{N,N}I_2 \\ 
		\end{smallmatrix}\!\right]\!\!, \ 
		\mathbb{1}_N\!=\!\left[\!\begin{smallmatrix} I_2  \\ \vdots \\ I_2\end{smallmatrix}\!\right]\!\!, \
		\bar{\mathbb 1}_N\!=\!\left[\!\begin{smallmatrix}  I_2  \\  \vdots \\(-1)^{N} I_2\end{smallmatrix}\!\right],
	\end{array}
\end{equation*}
and $\ell_{k,j} = (2j+1)(1 - (-1)^{j+k})$ if $ j \infeq k$ and $0$ otherwise.\\
Then system \eqref{eq:linWaveProblem} is $\alpha$-exponentially stable.
\end{theo}

The proof of this theorem relies on the construction of a Lyapunov functional described in the following subsections.

\begin{remark} A necessary condition for \eqref{eq:LMI} to be fulfilled is that the last $2 \times 2$ diagonal block of \eqref{eq:LMI} must be definite negative corresponding to the following inequality:
\[
	H^{\top} (S+R) H e^{2 \frac{\alpha}{c}} - G^{\top} S G \prec 0.
\]

This condition implies:
\begin{equation} \label{eq:alphaMax}
	\alpha \infeq \alpha_{max} = \max \left( \frac{c}{2} \log \left| \frac{(c k + 1)(c g + 1)}{(c k - 1)(c g - 1)} \right|, 0 \right).
\end{equation}

Setting $g = 0$ or $k=0$ leads to the same maximal decay-rate than in \cite{barreauInputOutput,bastin2016stability,datko1991two}. This condition is also related to the $\tau$-stabilization which is a common phenomenon when considering a wave equation \cite{olgac2004practical}. 
One can notice that for $g > 0$ and $k > 0$, the PDE system itself is asymptotically stable, because the two boundary conditions are adding damping. Notice that if one of them is negative, there exist also values of the other coefficient making the system asymptotically stable. Note also that for $g = c^{-1}$ or $k=c^{-1}$ leads to $\alpha_{max} = + \infty$ meaning there is no neutral part and the system resumes to a time-delay system. For $d > 0$, the neutral part is not modified and the same limit can be observed. \end{remark}


\begin{remark}[Hierarchy] \label{sec:hierarchy} Define the following set:
\[
	\mathcal{C}_N = \left\{ \alpha \supeq 0 \ | \ \Psi_{N,\alpha} - R_N \prec 0, P_N \succ 0, R \succ 0, S \succ 0 \right\},
\]
and assume this set is not empty. Then, denote $\alpha_{N} = \sup \mathcal{C}_N$. The hierarchy property states that $\alpha_{N+1} \supeq \alpha_{N}$. This can be proved using the same strategy than in \cite{besselString,SAFI20171}.
\end{remark}

\subsection{Proof of Theorem~\ref{sec:theoLin}}

\subsubsection{Preliminaries}

The main contribution of this paper relies on the extensive use of Bessel inequality to encompass traditional results. Before stating this inequality, we need to introduce an orthonormal family. The definition is as follows:
\begin{definition}[Legendre polynomials] Let $N \in \mathbb{N}$, the family of Legendre polynomials of degree less than or equal to $N$ is denoted by $\{\mathcal{L}_\ell\}_{\ell \in [0, N]}$ with 
\[
	\mathcal{L}_\ell(x) = (-1)^\ell \sum_{l = 0}^\ell (-1)^l \left( \begin{smallmatrix} \ell \\ l \end{smallmatrix} \right) \left( \begin{smallmatrix} \ell+l \\ l \end{smallmatrix} \right) x^l
\]
with $\left( \begin{smallmatrix} \ell \\ l \end{smallmatrix} \right) = \frac{\ell!}{l! (\ell-l)!}$.
\end{definition}

The sequence $\{\mathcal{L}_k\}$ is made up of ``shifted''-Legendre polynomials on $[0, 1]$. As seen in \cite{baudouin:hal-01310306,courant1966courant,seuret:hal-01065142}, this family is orthonormal in $L^2$ with the canonical inner product. That leads to the following definition.
\begin{definition} Let $\chi \in L^2$. The projection of $\chi$ on the $\ell^{th}$ Legendre polynomials is defined as follows:
\[
	\Chi_\ell := \int_0^1 \chi(x) \mathcal{L}_\ell(x) dx.
\]
\end{definition}

The Bessel inequality is obtained considering the previous definitions and the orthogonal property of the shifted-Legendre family.
\begin{lemma}[Bessel Inequality] \label{Bess}
	For any function $\chi \in L^2$ and symmetric positive matrix $R \in \mathbb S^2_+$, the following Bessel-like integral inequality holds for all $N\in \mathbb N$:
	\begin{equation}\label{eq:Bessel}
		\int_{0}^1 \chi^{\top}(x) R \chi(x) dx  \supeq \sum_{\ell=0}^{N} (2\ell+1)  \Chi_\ell^{\top} R \Chi_\ell.
	\end{equation}
\end{lemma}

This lemma and its short proof can be seen in \cite{besselString}.

The derivation of $\Chi_\ell$ along time is needed in the sequel. Lemma~3 from \cite{besselString} deals with this issue.
\begin{lemma}\label{lem:Chi_k}
	For any function $\chi \in L^2$, the following expression holds for any $N$ in $\mathbb N$ using notations  \eqref{eq:defTheo}:
	\begin{equation*}
		\left[ \begin{smallmatrix}  \dot \Chi_0 \\ \vdots \\ \dot \Chi_{N} \end{smallmatrix} \right] = c\mathbb 1_N\chi(1)-c\bar {\mathbb 1}_N\chi(0) -cL_N \left[ \begin{smallmatrix}  \Chi_0 \\ \vdots \\ \Chi_{N} \end{smallmatrix} \right].
	\end{equation*}
\end{lemma}

The link between $\alpha$-exponential stability and a Lyapunov functional is made by the following lemma.
\begin{lemma} \label{sec:lemmaEpsilon} Let $V$ be a Lyapunov functional for system \eqref{eq:linWaveProblem} and $\alpha \geq 0$. Assume there exist $\varepsilon_1, \varepsilon_2, \varepsilon_3 > 0$ such that the following holds for all $t \supeq 0$:
\begin{equation} \label{eq:epsilon}
	\left\{
	\begin{array}{l}
		 \varepsilon_1 \| (X, \chi) \|^2_{\mathcal H} \infeq V(X, \chi) \infeq \varepsilon_2 \| (X, \chi) \|^2_{\mathcal H}, \\
		\dot{V}(X, \chi) + 2 \alpha V(X, \chi) \infeq - \varepsilon_3 \| (X, \chi) \|^2_{\mathcal H},
	\end{array}
	\right.
\end{equation}
then system \eqref{eq:linWaveProblem} is $\alpha$-exponentially stable.
\end{lemma}
\begin{proof}
Inequalities \eqref{eq:epsilon} bring the following: $\dot{V}(X,w) + \left (\alpha + \frac{\varepsilon_3}{\varepsilon_2} \right)  V(X, \chi) \infeq 0$. Then integrating this inequality between $0$ and $t$ leads to:
	\begin{equation*}
		 \| (X(t), \chi(t) ) \|^2_{\mathcal H} \infeq \frac{\varepsilon_2}{\varepsilon_1} \| (X^0, \chi_0) \|^2_{\mathcal H} e^{- 2 \alpha t}.
	\end{equation*}
\end{proof}

Once these useful lemmas reminded, a Lyapunov functional can be defined. 
\subsubsection{Lyapunov functional candidate}

The aim of this subpart is to build a Lyapunov functional candidate for system \eqref{eq:linWaveProblem}. Following the same methodology than introduced in \cite{besselString}, a first Lyapunov functional $\V_{\alpha}$ for the PDE part is defined with $S, R \in \mathbb{S}^2_+$:
\[
	\V_{\alpha}(\chi) = \int_0^1 e^{2 \frac{\alpha x}{c}} \chi^{\top}(x) (S + xR) \chi(x) dx,
\]

The Lyapunov functional candidate is then the summation of a quadratic term and $\V_{\alpha}$. This quadratic term contains the stability of state $X$ but also some terms merging the ODE and the PDE. This is done to enlarge the stability analysis, enabling the study of stability of the whole interconnected system and not of each subsystem independently. This technique, as shown in \cite{besselString}, is well-suited for the study of an unstable ODE coupled with a PDE for instance. The total Lyapunov function of order $N \in \mathbb{N}$ is then:
\begin{equation} \label{eq:VN}
	V_{N, \alpha}(X, \chi) = X_N^{\top} P_N X_N + \V_{\alpha}(\chi)
\end{equation}
with $P_N \in \mathbb{S}^{n+2+2(N+1)}_+$ and $X_N = \left[ \begin{matrix} X^{\top} \!\!& \Chi_0^{\top} \!\!& \dots \!\!& \Chi_N^{\top} \end{matrix} \right]^{\top}\!\!$.

The aim now is to prove the existence of $\varepsilon_1, \varepsilon_2$ and $\varepsilon_3 > 0$ to apply Lemma~\ref{sec:lemmaEpsilon} on the functional $V_{N,\alpha}$ and then conclude the proof.
\subsubsection{Existence of $\varepsilon_1$}
Conditions $P_N \succ 0$ and $S, R \in \mathbb{S}^2_+$ mean that there exists $\varepsilon_1 > 0$, such that for all $x \in [0, 1]$:
	\vspace{-0.3cm}
	\[
		\begin{array}{rcl}
			P_N \!\!&\succeq& \!\!\!\varepsilon_1 \text{diag} \left( I_{n+2}, 0_2 \right) , \\
			S+xR \ \succeq \ S &\succeq& \!\!\!\frac{\varepsilon_1}{2}  I_2.
		\end{array}
	\]
	
	These inequalities imply:
	\[
		\begin{array}{lcl}
			V_{N,\alpha}(X, w) \!\!\! & \supeq & \!\!\!\varepsilon_1 \left( |X|^2 + \frac{1}{2} \|\chi\|^2 \right) \\
			&& + \int_0^1\chi^\top(x)\left(S + xR - \frac{\varepsilon_1}{2} I_2\right)\chi(x)dx\\
			&\supeq& \!\!\!\varepsilon_1 \left( |X|_n^2 + \frac{1}{2} \|\chi\|^2 \right) \supeq \varepsilon_1 \|(X, \chi)\|_{\mathcal{H}}^2.
		\end{array}
	\]

\subsubsection{Existence of $\varepsilon_2$}
	Since $P_N, S$ and $R$ are definite positive matrices, there exists $\varepsilon_2 > 0$ such that:
	\vspace{-0.1cm}
	\[
		\begin{array}{rcl}
			P_N \!\!\!& \preceq \!\!\!& \text{diag} \left( \varepsilon_2 I_{n+2}, \frac{\varepsilon_2}{4} \text{diag} \left\{ (2\ell+1) I_n \right\}_{\ell \in (0, N)} \right), \\
			(S+xR) \!\!\!& \preceq \!\!\!& S + R \ \preceq \ \frac{\varepsilon_2}{4} e^{-2 \frac{\alpha}{c}}I_2,\quad \forall x\in (0,1).
		\end{array}
	\]
	
	Then, we get:
	\vspace{-0.1cm}
	\begin{equation*}
		\begin{array}{lll}
			V_{N,\alpha}(X, \chi) 
				&\!\!\!\! \infeq &\!\!\!\! \displaystyle\varepsilon_2 |X|^2  \vphantom{\sum_{\ell=0}^{N}} \! +\! \frac{\varepsilon_2}{4}\! \left( \sum_{\ell=0}^{N} (2\ell\!+\!1) \Chi_\ell^{\top} \Chi_\ell + \| \chi \|^2 \right) \\
&\!\!\!\! \infeq &\!\!\!\! \varepsilon_2 \left( |X|^2\! +\! \frac{1}{2} \|\chi\|^2 \right) = \varepsilon_2 \|(X,\chi)\|_{\mathcal{H}}^2.
		\end{array}
	\end{equation*}
	
The inequality comes from Bessel inequality \eqref{eq:Bessel}.

\subsubsection{Existence of $\varepsilon_3$}

This part is the most important and shows that system \eqref{eq:linWaveProblem} is dissipative \cite{besselString,tucsnak2009observation}. Differentiating with respect to time \eqref{eq:VN} along the trajectories of system \eqref{eq:linWaveProblem} leads to:
	\begin{equation*}
		\dot{V}_{N,\alpha}(X, w) = \text{He} \left( \left[ \begin{smallmatrix} \dot{X} \\  \dot{\Chi}_0 \\ \vdots \\ \dot{\Chi}_N \end{smallmatrix} \right]^{\top} P_N \left[ \begin{smallmatrix} {X} \\  {\Chi}_0 \\ \vdots \\ {\Chi}_N\end{smallmatrix} \right]  \right) + \dot{\V}_{\alpha}(w).
	\end{equation*}

	The goal here is to find an upper bound of $\dot{V}_{N,\alpha}$ using the extended state: $\xi_N = \left[ X_N^{\top} \ \ w_t(1) \ \ w_t(0) \right]^{\top}$. The first step is to derive an expression of $\dot{\V}_{\alpha}$. Similarly to \cite{besselString}, we get:
	\[
		\hspace*{-0.12cm}
		\begin{array}{ll}
			\dot{\mathcal{V}}_{\alpha}(\chi) \!\!\!\!\!&= 2 c \int_0^1 \chi^{\top}_x(x) (S + xR) \chi(x) e^{2 \frac{\alpha x}{c}} dx \\
			&= 2c \left( \chi^{\top}(1) (S+R) \chi(1) e^{2 \frac{\alpha}{c}} - \chi^{\top}(0) S \chi(0) \right. \quad \quad \quad \\
			& \hfill \left.- \int_0^1 \chi^{\top}(x) R \chi(x) e^{2 \frac{\alpha x}{c}} dx \right) - 4 \alpha \V_{\alpha}(\chi) - \dot{\V}_{\alpha}(\chi) \\
			&= c \left( \chi^{\top}(1) (S+R) \chi(1) e^{2 \frac{\alpha}{c}} - \chi^{\top}(0) S \chi(0) \right. \quad \quad \quad \\
			& \hfill \left.- \int_0^1 \chi^{\top}(x) R \chi(x) e^{-2 \frac{\alpha x}{c}} dx \right) - 2 \alpha \V_{\alpha}(\chi).
		\end{array}
	\]
	
	Using the previous equation, Lemma~\ref{lem:Chi_k} and equation \eqref{eq:problem}, we note that $X_N = F_N\xi_N, ~ \dot X_N = Z_N\xi_N, ~ \chi(0) = G_{N}\xi_N, ~ \chi(1) = H_{N}\xi_N$ where matrices $F_N, Z_N, H_{N}, G_{N}$ are given in \eqref{eq:defTheo}. Then we can write:
	\vspace{-0.22cm}
	\begin{multline*}
		\dot{V}_{N,\alpha}(X, \chi) = \xi_N^{\top} \Psi_{N,\alpha} \xi_N + c \sum_{\ell=0}^{N} \Chi_\ell^{\top} (2\ell+1) R \Chi_\ell  \\
		-c  \int_0^1 \chi^{\top}(x) R \chi(x) e^{2 \frac{\alpha x}{c}} dx - 2 \alpha V_{N,\alpha}(X, \chi).
	\end{multline*}
	
Denoting by $W_{N,\alpha}(X, \chi) = \dot{V}_{N,\alpha}(X, \chi) + 2 \alpha V_{N, \alpha}(X, \chi)$, the previous equality implies the following upper bound:
\begin{multline} \label{eq:WN}
	W_{N,\alpha}(X, \chi) \infeq \xi_N^{\top} \Psi_{N,\alpha} \xi_N + c \sum_{\ell=0}^{N} (2\ell+1)  \Chi_\ell^{\top} R \Chi_\ell  \\
		-c  \int_0^1 \chi^{\top}(x) R \chi(x) dx.
\end{multline}
	
Since  $R \succ 0$ and $\Psi_{N,\alpha} \prec 0$, there exists $\varepsilon_3 > 0$ such that:
	\begin{equation} \label{eq:Rpos2} 
		\begin{array}{ccl}
			R \!\!\!&\succeq & \ \frac{\varepsilon_3}{2} I_2,\\
			\Psi_{N,\alpha} \!\!\!&\preceq &\!\!\! -\varepsilon_3 \text{diag} \left(  I_{n+2}, \frac{1}{2} I_2, \frac{3}{2}I_2,\dots, \frac{2N\!+\!1}{2} I_2, 0_2 \right)\!.
		\end{array}
	\end{equation}
	
Using \eqref{eq:Rpos2} and Bessel's inequality, equation \eqref{eq:WN} becomes:
	\begin{equation*}
		W_{N,\alpha}(X,\chi) \infeq \!\!- \varepsilon_3 \left( \!|X|^2 + \frac{1}{2} \| \chi \|^2 \!\right) \infeq \!\!- \varepsilon_3 \ \| (X, \chi) \|^2_{\mathcal H},
	\end{equation*}
	and that concludes the proof.

\section{Examples and discussion}

In this section, we illustrate the proposed theorem by using values taken from \cite{marquez2015analysis, saldivar2016control} and shown in Table \ref{tab:values}. The simulation is based on a finite-difference method of order $2$. The two cases under study here are summarized below:
\begin{enumerate}
	\item the feedforward control with $n = 0$ (using only $u_1^e$ and $u_2^e$ in \eqref{eq:feedforward}) and 
		\begin{equation} \label{eq:uncontrolled}
			\begin{array}{lll}
				C_1 = \left[ \begin{smallmatrix} 0 & 0 \end{smallmatrix} \right], & C_2 = 0, & K = \left[ \begin{smallmatrix} 0 & 0 \end{smallmatrix} \right].
			\end{array}
		\end{equation}	
	\item a dynamic control with the following parameters:
		\begin{equation} \label{eq:controlled}
			\hspace{-1cm}
			\begin{array}{lll}
				\tilde{A} = \left[ \begin{smallmatrix} -800 & 0 \\ 0 & -150 \end{smallmatrix} \right], & B_{c1} = 0_{2,2}, & \!\! B_{c2} = I_2, \\
				C_1 = \left[ \begin{smallmatrix} 800 & 0.015 & 0.01 & -0.1 \end{smallmatrix} \right], & C_2 = \left[ \begin{smallmatrix} 0 & - 0.0718 \end{smallmatrix} \right], \\
				 K = \left[ \begin{smallmatrix} -82.2 & \ 10.4 \end{smallmatrix} \right].
			\end{array}
			\hspace{-1cm}
	\end{equation}
\end{enumerate}

The dynamic controller is obtained considering two low-pass filters. Denote by $s \in \mathbb{C}$ the Laplace variable, the two transfer functions for the low-pass filters are $\frac{u_1}{w_t(0)} = \frac{1}{1 + s \omega_{c1}}$ and $\frac{u_1}{w_t(1)} = \frac{1}{1 + \omega_{c2}}$ with the cut-off frequencies $\omega_{c1} = 800$ and $w_{c2} = 150$. Gain $K$ has been chosen such that the eigenvalues of $A+BK$ are $-2.4603 \pm 0.1230 i$. $C_2$ has been chosen to cancel the dependence on $w_t(1,\cdot)$ in the ODE.

With the feedforward controller only, it is possible to estimate the decay-rate of the solution. Indeed, there is no real coupling between the ODE and the PDE and the decay-rate of the interconnected system will be the smallest between their respective ones. Here, the PDE has a decay-rate given by equation \eqref{eq:alphaMax} of $1.2302$ and the ODE is $0.2159$. The results of Theorem~\ref{sec:theoLin} is given in Table~\ref{tab:alphaMax}. The maximum decay-rate for the feedforward case is obtained for $N \supeq 1$ and is, as expected, the decay-rate of the ODE.

Figure~\ref{fig:simu} shows the time response of system \eqref{eq:linWaveProblem} in the two cases. The initial state for this computation is $X^0 = 0, w(x,0) = 2 - \Omega_e x$ and $w_{t}(x,0) = \frac{\Omega_e - q T^e}{k} x - \frac{u_1^e - \Omega_e}{g} (1-x)$ for $x \in (0, 1)$. Of course states $X_1$ and $X_2$ are much faster, which results from the direct influence of static feedback gain $K$ but also the speed $w_t(1)$, which is more regular and converges faster to $0$. Indeed, as shown in Table \ref{tab:alphaMax}, the speed is much faster in the situation with the dynamic control. The hierarchy of Remark~\ref{sec:hierarchy} is clearly visible and reaches its maximum value (up to three a 3 digits precision) at $N = 2$. If $d > 0$, one can notice a slightly higher decay rate but the limit remains the same. One of the drawback of such a system is the angular speed $w_t(x,\cdot)$ for $x \in (0,1)$, which increases significantly compared to the first case as it is possible to see on Figure~\ref{fig:simu3D}. 

\begin{remark} A backstepping control law could have been considered with a target system of arbitrary large decay-rate. Compared to this method, the price to pay for a finite dimension controller is seen by equation~\eqref{eq:alphaMax}. Indeed, it is not possible to accelerate the system with an arbitrary large decay-rate. Other differences are that there is no design methodology using LMI yet and the control is a finite-dimension state-feedback using the knowledge of only $Y$, $w_t(0)$ and $w_t(1)$ with strictly proper controllers.\end{remark}

\begin{table}
	\centering
	\begin{tabular}{c|c||c|c}
		Symbol & Value & Symbol & Value \\
		\hline
		$c$ & $2.6892$ m.s${}^{-1}$ & $\Omega_e$ & $10$ rad.s${}^{-1}$  \\
		$k$ & $0.1106$ s.m${}^{-1}$ & $g$ & $2.48$ s.m${}^{-1}$ \\
		$A_{21}$ & $-41.58$ s${}^{-2}$ & $A_{22}$ & $-0.43$ s${}^{-1}$\\
		$e_1$ & $-8.35$ m.s${}^{-1}$.rad${}^{-1}$ &$e_2$ & $-0.069$ m${}^{-1}$.kg${}^{-1}$ \\
		$b$ & $-0.43$ s${}^{-1}$ & $T^e$ & $7572.4$ N.m \\
		$q$ & $0.0012$ N${}^{-1}$.m${}^{-1}$
	\end{tabular}
	\vspace{0.2cm}
	\caption{Coefficient values taken for the simulations.}
	\label{tab:values}
\end{table}

\begin{table}
	\centering
	\begin{tabular}{c|ccccc}
		Type of control& $N = 0$ & $N = 1$ & $N = 2$ & $N =3$ & $\alpha_{max}$ \\
		\hline
		Feedforward & $0.2157$ & $0.2159$ & $0.2159$ & $0.2159$ & $1.2302$ \\
		Dynamic & $0.4972$ & $0.4972$ & $1.000$ & $1.000$ & $1.2302$
	\end{tabular}
	\vspace{0.1cm}
	\caption{Maximum decay-rate $\alpha$ using Theorem~\ref{sec:theoLin} at an order $N$. The feedforward controller refers to \eqref{eq:uncontrolled} while the dynamic controller is with \eqref{eq:controlled}. $\alpha_{max}$ is calculated using \eqref{eq:alphaMax}.}
	\label{tab:alphaMax}
\end{table}


\begin{figure*}
	\centering
	\includegraphics[trim = 2.75cm 0cm 2.75cm 0cm, width=18cm]{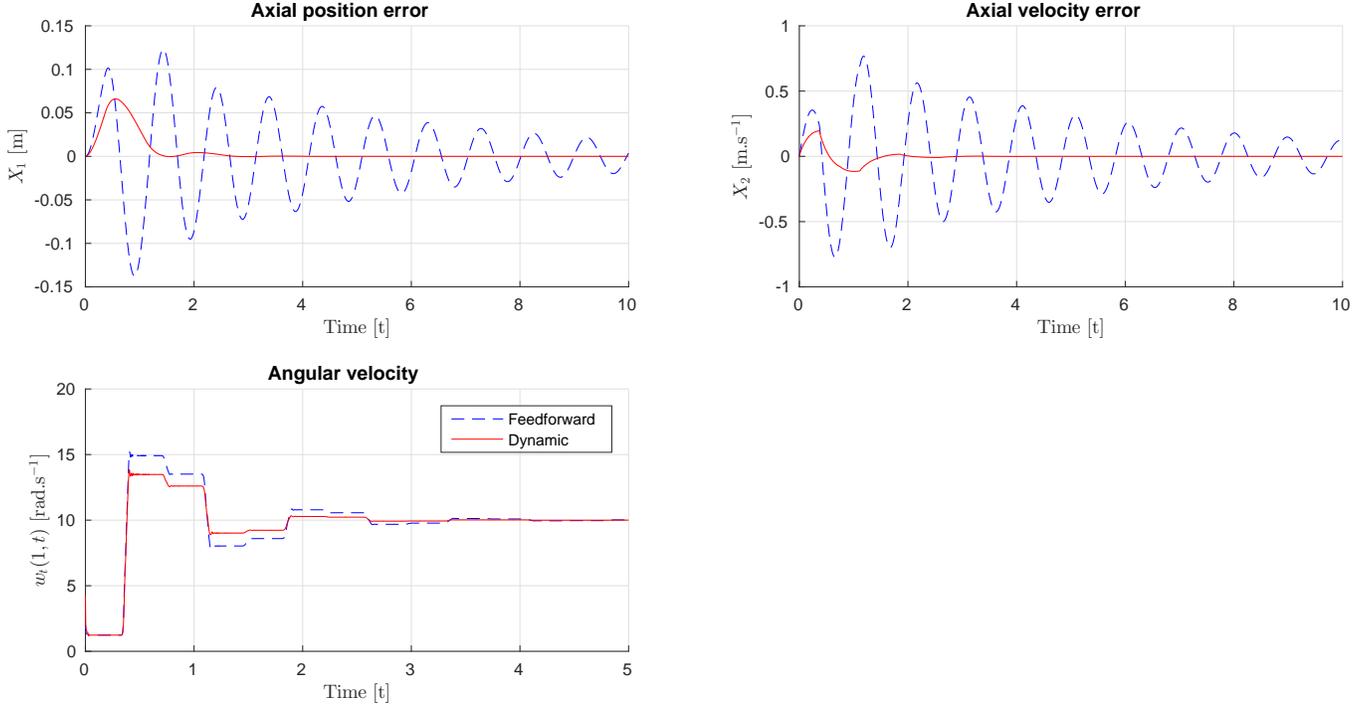}
	\caption{Simulation on the feedforward and dynamic controlled system. }
	\label{fig:simu}
\end{figure*}

\begin{figure}
	\centering
	\includegraphics[trim = 0.5cm 0cm 0.5cm 0cm, clip, width=8.5cm]{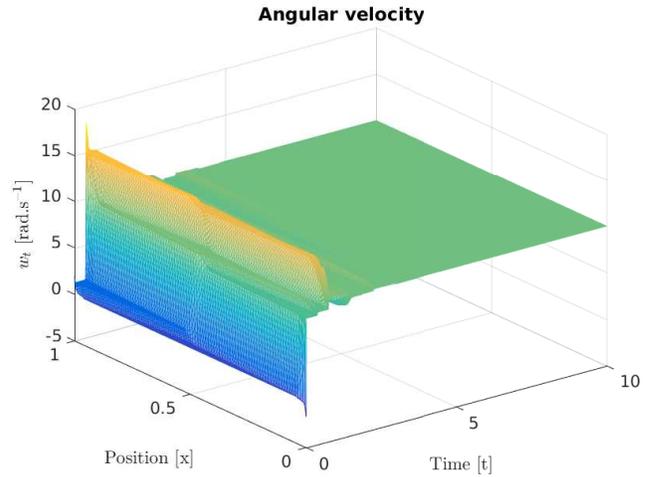}
	\caption{Angle velocity $w_t$ in the situation with dynamic control.}
	\label{fig:simu3D}
\end{figure}

\section{Conclusion}
We have studied the stability of a drilling mechanism, which dynamics can be modeled as a coupled ODE/PDE. Approximating this model around a desired equilibrium point leads to an interconnected ODE / damped wave equation. Therefore, the stability of this coupled system is studied using a Lyapunov approach and the stability condition of such a system has been expressed in terms of LMI. Using Bessel inequality, we provided a hierarchy of LMI conditions for this kind of interconnected system with linear feedback controllers. Using only strictly proper hand-designed controllers, a control law has been derived improving subsequently the decay-rate of the system. Further studies would investigate how to automatically design such controllers.

\section{ACKNOWLEDGMENTS}

The authors gratefully acknowledge anonymous reviewers' comments. This work is supported by the ANR project SCIDiS contract number 15-CE23-0014.

\bibliographystyle{plain}
\bibliography{report_draft}

\begin{thebibliography}{10}

\bibitem{barreauInputOutput}
M.~Barreau, F.~Gouaisbaut, A.~Seuret, and R.~Sipahi.
\newblock Input / output stability of a damped string equation coupled with
  ordinary differential system.
\newblock Working paper, available on HAL, 2018.

\bibitem{besselString}
M.~Barreau, A.~Seuret, F.~Gouaisbaut, and L.~Baudouin.
\newblock {Lyapunov stability analysis of a string equation coupled with an
  ordinary differential system}.
\newblock {\em IEEE Trans. Automatic Control}, 2018.
\newblock available on HAL.

\bibitem{bastin2016stability}
G.~Bastin and J.-M. Coron.
\newblock {\em Stability and boundary stabilization of 1-d hyperbolic systems},
  volume~88.
\newblock Springer, 2016.

\bibitem{baudouinHeat}
L.~Baudouin, A.~Seuret, and F.~Gouaisbaut.
\newblock Lyapunov stability analysis of a linear system coupled to a heat
  equation.
\newblock In {\em 20th IFAC World Congress, Toulouse}, volume~50, pages 11978
  -- 11983, 2017.

\bibitem{baudouin:hal-01310306}
L.~Baudouin, A.~Seuret, and M.~Safi.
\newblock Stability analysis of a system coupled to a transport equation using
  integral inequalities.
\newblock volume~49, pages 92 -- 97, 2016.
\newblock 2nd {IFAC} Workshop on Control of Systems Governed by Partial
  Differential Equations ({CPDE}).

\bibitem{bresch2014output}
D.~Bresch-Pietri and M.~Krstic.
\newblock Output-feedback adaptive control of a wave {PDE} with boundary
  anti-damping.
\newblock {\em Automatica}, 2014.

\bibitem{challamel2000rock}
N.~Challamel.
\newblock Rock destruction effect on the stability of a drilling structure.
\newblock {\em Journal of sound and vibration}, 233(2):235--254, 2000.

\bibitem{courant1966courant}
R.~Courant and D.~Hilbert.
\newblock {\em Methods of mathematical physics}.
\newblock John Wiley \& Sons, Inc., 1989.

\bibitem{datko1991two}
R~Datko.
\newblock Two questions concerning the boundary control of certain elastic
  systems.
\newblock {\em Journal of Differential Equations}, 92(1):27--44, 1991.

\bibitem{fridman2014}
E.~Fridman.
\newblock {\em Introduction to Time-Delay Systems}.
\newblock Analysis and Control. Birkh{\"a}user, 2014.

\bibitem{fridman2010bounds}
E.~Fridman, S.~Mondi{\'e}, and B.~Saldivar.
\newblock Bounds on the response of a drilling pipe model.
\newblock {\em IMA Journal of Mathematical Control and Information},
  27(4):513--526, 2010.

\bibitem{hale2001effects}
J.~K. Hale and S.~M.~V. Lunel.
\newblock Effects of small delays on stability and control.
\newblock In {\em Operator theory and analysis}, pages 275--301. Springer,
  2001.

\bibitem{he2014adaptive}
W.~He, S.~Zhang, and S.~S. Ge.
\newblock Adaptive control of a flexible crane system with the boundary output
  constraint.
\newblock {\em IEEE Transactions on Industrial Electronics}, 61(8):4126--4133,
  2014.

\bibitem{krstic2009delay}
M.~Krstic.
\newblock {\em Delay compensation for nonlinear, adaptive, and {PDE} systems}.
\newblock Springer, 2009.

\bibitem{krstic2008boundary}
M.~Krstic and A.~Smyshlyaev.
\newblock {\em Boundary control of PDEs: A course on backstepping designs},
  volume~16.
\newblock Siam, 2008.

\bibitem{morgul1994dynamic}
{\"O}.~Morg{\"u}l.
\newblock A dynamic control law for the wave equation.
\newblock {\em Automatica}, 30(11):1785--1792, 1994.

\bibitem{olgac2004practical}
N.~Olgac and R.~Sipahi.
\newblock A practical method for analyzing the stability of neutral type
  {LTI}-time delayed systems.
\newblock {\em Automatica}, 40(5):847--853, 2004.

\bibitem{SAFI20171}
M.~Safi, L.~Baudouin, and A.~Seuret.
\newblock Tractable sufficient stability conditions for a system coupling
  linear transport and differential equations.
\newblock {\em Systems \& Control Letters}, 110:1 -- 8, 2017.

\bibitem{marquez2015analysis}
B.~Saldivar, I~Boussaada, H.~Mounier, and S.-I. Niculescu.
\newblock {\em Analysis and Control of Oilwell Drilling Vibrations: A
  Time-Delay Systems Approach}.
\newblock Springer, 2015.

\bibitem{saldivar2016control}
B.~Saldivar, S.~Mondi{\'e}, and J.~C. {\'A}vila~Vilchis.
\newblock The control of drilling vibrations: A coupled {PDE}-{ODE} modeling
  approach.
\newblock {\em International Journal of Applied Mathematics and Computer
  Science}, 2016.

\bibitem{seuret:hal-01065142}
A.~Seuret and F.~Gouaisbaut.
\newblock Hierarchy of {LMI} conditions for the stability analysis of time
  delay systems.
\newblock {\em Systems \& Control Letters}, 81:1--7, 2015.

\bibitem{ATJECC}
A.~Terrand-Jeanne, V.~Dos Santo~Martins, and V.~Andrieu.
\newblock Regulation of the downside angular velocity of a drilling string with
  a {P-I} controller.
\newblock {\em ECC 2018, Cyprus}, pages 2647--2652, 2018.

\bibitem{tucsnak2009observation}
M.~Tucsnak and G.~Weiss.
\newblock {\em Observation and control for operator semigroups}.
\newblock Springer, 2009.

\bibitem{Wu20142787}
H.-N. Wu and J.-W. Wang.
\newblock Static output feedback control via {PDE} boundary and {ODE}
  measurements in linear cascaded {ODE}-beam systems.
\newblock {\em Automatica}, 50(11):2787--2798, 2014.

\end{thebibliography}

\end{document}